\documentclass[12pt,reqno]{article}

\usepackage[usenames]{color}
\usepackage[colorlinks=true,
linkcolor=webgreen, filecolor=webbrown,
citecolor=webgreen]{hyperref}

\definecolor{webgreen}{rgb}{0,.5,0}
\definecolor{webbrown}{rgb}{.6,0,0}

\usepackage{amssymb}
\usepackage{graphicx}
\usepackage{amscd}
\usepackage{lscape}
\usepackage{tikz}
\usepackage{tikz-cd}
\usepackage{pgfplots}

\usetikzlibrary{matrix}
\usetikzlibrary{fit,shapes}
\usetikzlibrary{positioning, calc}
\tikzset{circle node/.style = {circle,inner sep=1pt,draw, fill=white},
        X node/.style = {fill=white, inner sep=1pt},
        dot node/.style = {circle, draw, inner sep=5pt}
        }
\usepackage{tkz-fct}

\usepackage{amsthm}
\newtheorem{theorem}{Theorem}

\newtheorem{proposition}[theorem]{Proposition}
\newtheorem{corollary}[theorem]{Corollary}

\theoremstyle{definition}

\newtheorem{example}[theorem]{Example}

\usepackage{float}

\usepackage{graphics,amsmath}
\usepackage{amsfonts}
\usepackage{latexsym}
\usepackage{epsf}

\begin{document}

\begin{center}
\vskip 1cm{\LARGE\bf On pseudo-involutions, involutions and quasi-involutions in the group of almost Riordan arrays} \vskip 1cm \large
Paul Barry, Nikolaos Pantelidis\\
School of Science\\
Waterford Institute of Technology\\
Ireland\\
\href{mailto:pbarry@wit.ie}{\tt pbarry@wit.ie}\\
\href{mailto:nikolaospantelidis@gmail.com}{\tt nikolaospantelidis@gmail.com}
\end{center}
\vskip .2 in

\begin{abstract} \noindent The group of almost Riordan arrays contains the group of Riordan arrays as a subgroup. In this note, we exhibit examples of pseudo-involutions, involutions and quasi-involutions in the group of almost Riordan arrays \end{abstract}

\section{Introduction}

Since the initial paper \cite{SGWW} defining the Riordan group, there has been interest in studying involutions, pseudo-involutions and quasi-involutions \cite{Cameron, CKS, quasi, Jean_Louis} associated to this group. In this note, we take a look at involutions, pseudo-involutions and quasi-involutions associated to the related group of almost Riordan arrays \cite{almost}.

In this first section, we recall the definition of the Riordan group, and in the next section, we recall the definition of the group of almost Riordan arrays (of order $1$, initially). We then proceed to look at involutions, pseudo-involutions and quasi-involutions associated this group.

We define $\mathcal{F}_r$ to be the set of formal power series of order $r$, where
$$\mathcal{F}_r = \{ a_r x^r+a_{r+1}x^{r+1}+ a_{r+2}x^{r+2}+\ldots | a_i \in R\}$$
 where $R$ is a suitable ring with unit (which we shall denote by $1$). In the sequel, we shall take $R=\mathbb{Z}$.
The Riordan group over $R$ is then given by the semi-direct product
$$\mathcal{R}=\mathcal{F}_0 \rtimes \mathcal{F}_1.$$
To an element $(g, f) \in \mathcal{R}$ we associate the $R$-matrix with $(n,k)$-th element
$$T_{n,k} = [x^n] g(x) f(x)^k.$$
This is an invertible lower-triangular matrix. We shall also use the notation $(g(x), f(x))$ to represent this matrix.
\begin{example} The Riordan array $\left(\frac{1}{1-x}, \frac{x}{1-x}\right)$ has associated matrix equal to the binomial matrix (Pascal's triangle) $B=\left(\binom{n}{k}\right)$. We have
\begin{align*}
T_{n,k}&=[x^n] \frac{1}{1-x}\frac{x^k}{(1-x)^k}\\
&=[x^{n-k}] \frac{1}{(1-x)^{k+1}}\\
&=[x^{n-k}] \binom{-(k+1)}{j}(-1)^j x^j\\
&=[x^{n-k}] \binom{k+1+j-1}{j} x^j\\
&=\binom{k+n-k}{n-k}\\
&=\binom{n}{k}.
\end{align*}
\end{example}

We shall confine ourselves to the subgroup
$$\mathcal{R}^{(1)}=\mathcal{F}_0^{(1)} \rtimes \mathcal{F}_1^{(1)},$$
where $$g(x) \in \mathcal{F}_0^{(1)} \leftrightarrow g_0=1,$$ and
$$f(x) \in \mathcal{F}_1^{(1)} \leftrightarrow f_1=1.$$
This ensures that all elements of $(g,f)^{-1}$ are in $R$, where
$$(g,f)^{-1}=\left(\frac{1}{g(\bar{f})}, \bar{f}\right).$$ Here, $\bar{f}(x)$ is the compositional inverse of $f(x)$.
The multiplication in the Riordan group is specified by
$$ (g(x), f(x)) \cdot (u(x), v(x)) = (g(x) u(f(x)), v(f(x))).$$ The element $I=(1,x)$ is the identity for multiplication. The fundamental theorem of Riordan arrays says that if $h(x) \in R[[x]]$, then
$$(g, f) h(x) = g(x)h(f(x)).$$

\subsection{Almost Riordan arrays}
We have shown \cite{almost} that if we define
$$a\mathcal{R}(1)=\{(a; g, f) | a \in \mathcal{F}_0, (g, f) \in \mathcal{R}\},$$
then $a\mathcal{R}(1)$ is also a group, called the group of almost Riordan arrays (of order $1$). In this case $a(x)$ is the generating function of the initial column of the almost-Riordan array. Such a matrix then begins
$$\left(
\begin{array}{cccc}
 a_0 & 0 & 0 & 0 \\
 a_1 & g_0 & 0 & 0 \\
 a_2 & g_1 & g_0f_1 & 0 \\
 a_3 & g_2 & g_0f_2+g_1 f_1 & 1 \\
\end{array}
\right).$$
Here, the Riordan array $(g(x), f(x))$ begins
$$\left(
\begin{array}{ccc}
 g_0 & 0 & 0 \\
 g_1 & g_0f_1 & 0 \\
 g_2 & g_0f_2+g_1f_1 & 1 \\
\end{array}
\right).$$
Elements of $a\mathcal{R}(1)$ are called almost Riordan arrays of the first order.

We recall that the product in the group of almost Riordan arrays (of first order) is given by
$$(a; g, f)\cdot (b; u,v) = ((a;g,f)b; g u(f), v(f)),$$ where
the fundamental theorem of almost Riordan arrays tells us that
$$(a; g, f) h(x) = h_0 a+ x g \tilde{h}(f),$$ with
$$\tilde{h}(x)=\frac{h(x)-h_0}{x}.$$
The inverse in $a\mathcal{R}(1)$ is given by
$$(a; g, f)^{-1}= \left(a^*; \frac{1}{g(\bar{f})}, \bar{f}\right),$$ where
$$a^*(x)  = (1; -g, f)^{-1}a(x),$$ with
$$(1; g, f)^{-1}=\left(1; \frac{1}{g(\bar{f})}, \bar{f}\right).$$
To the almost Riordan array $(a; g, f)$ we associate the matrix $M$ with
\begin{align*}
M_{n,k}&=[x^{n-1}] g f^{k-1},\quad n,k \ge 1,\\
M_{n,0}&=a_n,\\
M_{0,k}&=a_0 0^k. \end{align*}

\noindent We can identity the normal subgroup of elements of the form $(a;1,x)$ with  $\mathcal{F}_0$, and hence we have
$$a\mathcal{R}(1)=\mathcal{F}_0 \rtimes \mathcal{R}.$$ For instance, we have
$$\left(
\begin{array}{ccccccc}
 1 & 0 & 0 & 0 & 0 & 0 & 0 \\
 a_1 & 1 & 0 & 0 & 0 & 0 & 0 \\
 a_2 & 0 & 1 & 0 & 0 & 0 & 0 \\
 a_3 & 0 & 0 & 1 & 0 & 0 & 0 \\
 a_4 & 0 & 0 & 0 & 1 & 0 & 0 \\
 a_5 & 0 & 0 & 0 & 0 & 1 & 0 \\
 a_6 & 0 & 0 & 0 & 0 & 0 & 1 \\
\end{array}
\right)\cdot \left(
\begin{array}{ccccccc}
 1 & 0 & 0 & 0 & 0 & 0 & 0 \\
 0 & 1 & 0 & 0 & 0 & 0 & 0 \\
 0 & 2 & 1 & 0 & 0 & 0 & 0 \\
 0 & 3 & 3 & 1 & 0 & 0 & 0 \\
 0 & 4 & 6 & 4 & 1 & 0 & 0 \\
 0 & 5 & 10 & 10 & 5 & 1 & 0 \\
 0 & 6 & 15 & 20 & 15 & 6 & 1 \\
\end{array}
\right)$$
$$=\left(
\begin{array}{ccccccc}
 1 & 0 & 0 & 0 & 0 & 0 & 0 \\
 a_1 & 1 & 0 & 0 & 0 & 0 & 0 \\
 a_2 & 2 & 1 & 0 & 0 & 0 & 0 \\
 a_3 & 3 & 3 & 1 & 0 & 0 & 0 \\
 a_4 & 4 & 6 & 4 & 1 & 0 & 0 \\
 a_5 & 5 & 10 & 10 & 5 & 1 & 0 \\
 a_6 & 6 & 15 & 20 & 15 & 6 & 1 \\
\end{array}
\right).$$
Whilst all matrices in this note are of infinite extent for $n,k \ge 0$ (except for the last section) we display only indicative truncations.

\section{Pseudo-involutions in the group of almost Riordan arrays of first order}
We let $\bar{I}=(1,-x)$. We have $\bar{I}^2=I$.  We say that $(g, f) \in \mathcal{R}$ is an pseudo-involution in the Riordan group if $((g,f)\bar{I})^2=I$. The matrix corresponding to $\bar{I}$ is the diagonal matrix with diagonal $(1,-1,1,-1,\ldots)$.
We shall say that $(a; g,f) \in a\mathcal{R}(1)$ is a pseudo-involution in the group of almost Riordan arrays (of order 1) if $(M\bar{I})^2=I$, where $M$ is the matrix associated to $(a; g, f)$. Equivalently we require that $\bar{I}M\bar{I}=M^{-1}$.
We then have the following proposition.
\begin{proposition} The almost Riordan matrix $\left(\frac{1+x(r-1)}{1-x}; \frac{1}{(1-x)^2}, \frac{x}{1-x}\right)$ is a pseudo-involution in the group of almost Riordan arrays.
\end{proposition}

We note that the corresponding matrix $M$ coincides with the binomial matrix $\left(\frac{1}{1-x}, \frac{x}{1-x}\right)$, except that the initial column $(1,1,1,\ldots)^T$ has been replaced by $(1,r,r,r,\ldots)^T$. For $r \ne 1$ this is not a Riordan matrix, but an almost Riordan matrix. It begins
$$\left(
\begin{array}{ccccccc}
 1 & 0 & 0 & 0 & 0 & 0 & 0 \\
 r & 1 & 0 & 0 & 0 & 0 & 0 \\
 r & 2 & 1 & 0 & 0 & 0 & 0 \\
 r & 3 & 3 & 1 & 0 & 0 & 0 \\
 r & 4 & 6 & 4 & 1 & 0 & 0 \\
 r & 5 & 10 & 10 & 5 & 1 & 0 \\
 r & 6 & 15 & 20 & 15 & 6 & 1 \\
\end{array}
\right).$$
\begin{proof}
We must show that the initial column of the inverse matrix is $(1,-r,r,-r,\ldots)^T$, with generating function $1-\frac{r}{1+x}$. The ``interior'' elements are taken care of by the fact that the Riordan array $\left(\frac{1}{(1-x)^2},\frac{x}{1-x}\right)$ is a pseudo-involution in the Riordan group.
Now since
$$(a; g, f)^{-1} = \left(a^*; \frac{1}{g(\bar{f})}, \bar{f}\right),$$ we must calculate
$$a^*(x)  = (1; -g, f)^{-1}a(x),$$ where
$$a=\frac{1+x(r-1)}{1-x},\quad g(x)=\frac{1}{(1-x)^2}, \quad \text{and}\,f(x)=\frac{x}{1-x}.$$
We have
\begin{align*}
a^* &=(1; -g, f)^{-1} a(x)\\
&=\left(1; -\frac{1}{g(\bar{f})}, \bar{f}\right) a(x)\\
&=\left(1; -\frac{1}{(1+x)^2}, \frac{x}{1+x}\right) \frac{1+x(r-1)}{1-x}\\
&=a_0.1+x \left(\frac{-1}{(1+x)^2}\right)\frac{r}{1-\frac{x}{1+x}}\\
&=1-\frac{x}{(1+x)^2} \frac{r(1+x)}{1+x-x}\\
&=1-\frac{rx}{1+x}.
\end{align*}
\end{proof}
The inverse matrix thus begins
$$\left(
\begin{array}{ccccccc}
 1 & 0 & 0 & 0 & 0 & 0 & 0 \\
 -r & 1 & 0 & 0 & 0 & 0 & 0 \\
 r & -2 & 1 & 0 & 0 & 0 & 0 \\
 -r & 3 & -3 & 1 & 0 & 0 & 0 \\
 r & -4 & 6 & -4 & 1 & 0 & 0 \\
 -r & 5 & -10 & 10 & -5 & 1 & 0 \\
 r & -6 & 15 & -20 & 15 & -6 & 1 \\
\end{array}
\right).$$

We shall use the notation $A_r=\left(\frac{1+x(r-1)}{1-x}; \frac{1}{(1-x)^2}, \frac{x}{1-x}\right)$.
\begin{proposition} We have
$$A_r^p= \left(\frac{1+pr(x-1)}{1-px}; \frac{1}{(1-px)^2}, \frac{x}{1-px}\right).$$
\end{proposition}
Thus $A_r^p$ coincides with $B^p$, except the initial column $(1,p,p^2,p^3,\ldots)^T$ is replaced by $(1,rp,rp^2, rp^3,\ldots)$.
\begin{proposition} $A_r^p$ is a pseudo-involution in the group of almost Riordan matrices.
\end{proposition}

\begin{proof} We wish to show that the initial column of $(A_r^p)^{-1}$ is given by $(1, -rp, rp^2, -rp^3,\ldots)$, with generating function $1-\frac{prx}{1+px}$.
We have in this case that
\begin{align*}
a^* &=(1; -\frac{1}{(1-px)^2}, \frac{x}{1-px})^{-1} \frac{1+pr(x-1)}{1-px}\\
&=\left(1; -\frac{1}{(1+px)^2}, \frac{x}{1+px}\right) \frac{1+pr(x-1)}{1-px}\\
&=a_0.1+x \left(\frac{-1}{(1+px)^2}\right)\frac{pr}{1-p\frac{x}{1+px}}\\
&=1-\frac{x}{(1+px)^2} \frac{pr(1+px)}{1+px-px}\\
&=1-\frac{prx}{1+px}.
\end{align*}
\end{proof}

\begin{corollary} For each $r \ne 1$, the matrix $A_r$ generates a subgroup of pseudo-involutions in the group of almost Riordan arrays of first order.
\end{corollary}

We note that we have
$$A_r^p \cdot A_s^q = \left(1+\frac{prx}{1-px}+sx\left(\frac{1}{1-(p+q)x}-\frac{1}{1-px}\right);\frac{1}{(1-(p+q)x)^2}, \frac{x}{1-(p+q)x}\right).$$
Thus $A_r^p \cdot A_s^q \ne A_s^q \cdot A_r^p$ in general.

We end this section by noting that if $(g,f)$ is a pseudo-involution in the Riordan group, then $(1; g, f)$ is a pseudo-involution in the group of almost Riordan arrays. We have
$$\left(
\begin{array}{cccccc}
 1 & 0 & 0 & 0 & 0 & 0 \\
 r & 1 & 0 & 0 & 0 & 0 \\
 r & 2 & 1 & 0 & 0 & 0 \\
 r & 3 & 3 & 1 & 0 & 0 \\
 r & 4 & 6 & 4 & 1 & 0 \\
 r & 5 & 10 & 10 & 5 & 1 \\
\end{array}
\right)=\left(
\begin{array}{cccccc}
 1 & 0 & 0 & 0 & 0 & 0 \\
 r & 1 & 0 & 0 & 0 & 0 \\
 r & 0 & 1 & 0 & 0 & 0 \\
 r & 0 & 0 & 1 & 0 & 0 \\
 r & 0 & 0 & 0 & 1 & 0 \\
 r & 0 & 0 & 0 & 0 & 1 \\
\end{array}
\right) \cdot \left(
\begin{array}{cccccc}
 1 & 0 & 0 & 0 & 0 & 0 \\
 0 & 1 & 0 & 0 & 0 & 0 \\
 0 & 2 & 1 & 0 & 0 & 0 \\
 0 & 3 & 3 & 1 & 0 & 0 \\
 0 & 4 & 6 & 4 & 1 & 0 \\
 0 & 5 & 10 & 10 & 5 & 1 \\
\end{array}
\right),$$ where the last matrix $\left(1; \frac{1}{(1-x)^2}, \frac{x}{1-x}\right)$ is a pseudo-involution.
Taking inverses, we obtain
$$\left(
\begin{array}{cccccc}
 1 & 0 & 0 & 0 & 0 & 0 \\
 -r & 1 & 0 & 0 & 0 & 0 \\
 r & -2 & 1 & 0 & 0 & 0 \\
 -r & 3 & -3 & 1 & 0 & 0 \\
 r & -4 & 6 & -4 & 1 & 0 \\
 -r & 5 & -10 & 10 & -5 & 1 \\
\end{array}
\right)=\left(
\begin{array}{cccccc}
 1 & 0 & 0 & 0 & 0 & 0 \\
 0 & 1 & 0 & 0 & 0 & 0 \\
 0 & -2 & 1 & 0 & 0 & 0 \\
 0 & 3 & -3 & 1 & 0 & 0 \\
 0 & -4 & 6 & -4 & 1 & 0 \\
 0 & 5 & -10 & 10 & -5 & 1 \\
\end{array}
\right) \cdot \left(
\begin{array}{cccccc}
 1 & 0 & 0 & 0 & 0 & 0 \\
 -r & 1 & 0 & 0 & 0 & 0 \\
 -r & 0 & 1 & 0 & 0 & 0 \\
 -r & 0 & 0 & 1 & 0 & 0 \\
 -r & 0 & 0 & 0 & 1 & 0 \\
 -r & 0 & 0 & 0 & 0 & 1 \\
\end{array}
\right).$$

\section{Pseudo-involutions in the group of almost Riordan arrays of the second order}
It is possible \cite{almost} to extend the definition of $a\mathcal{R}(1)$ to derive the group $a\mathcal{R}(2)$ of almost Riordan arrays of order $2$. The elements of this group are of the form
$$(a,b;g,f)$$ where $a_0=1$, $b_0=1$, $g_0=1$ and $f_0=0, f_1=1$. The fundamental theorem for this group states that
$$(a, b; g, f)h=h_0 a+ h_1 x b+x^2 g \tilde{\tilde{h}}(f),$$ where
$$\tilde{\tilde{h}} = \frac{h(x)-h_0-h_1x}{x^2}.$$

\noindent Following \cite{almost}, we can define the inverse of a $4$-tuple as follows.
\begin{equation}(a, b; g, f)^{-1}=\left(a^{**}, b^*; \frac{1}{g \circ \bar{f}}, \bar{f}\right),\end{equation}
where \begin{equation}b^*=(1; - g, f)^{-1}\cdot b = \left(1; -\frac{1}{g \circ \bar{f}}, \bar{f}\right)\cdot b,\end{equation}
and where
\begin{equation}a^{**} = (1, -b; -g, f)^{-1} \cdot a=\left(1, -b^*; -\frac{1}{g \circ \bar{f}}, \bar{f}\right) \cdot a.\end{equation}

\begin{example} We define an element of $a\mathcal{R}(2)$ by $\left(\frac{1+x}{1-x}, \frac{1+x}{1-x}; \frac{1}{(1-x)^2}, \frac{x}{1-x}\right)$. The corresponding matrix $M$ is then defined by
\begin{align*}
M_{n,k}&=[x^n] x^k \frac{1+x}{1-x}, \quad k < 2;\\
M_{n,k}&=[x^n] \frac{x^k}{(1-x)^k}, \quad k \ge 2.\end{align*}
Thus the matrix $M$ begins
$$\left(
\begin{array}{ccccccc}
 1 & 0 & 0 & 0 & 0 & 0 & 0 \\
 2 & 1 & 0 & 0 & 0 & 0 & 0 \\
 2 & 2 & 1 & 0 & 0 & 0 & 0 \\
 2 & 2 & 2 & 1 & 0 & 0 & 0 \\
 2 & 2 & 3 & 3 & 1 & 0 & 0 \\
 2 & 2 & 4 & 6 & 4 & 1 & 0 \\
 2 & 2 & 5 & 10 & 10 & 5 & 1 \\
\end{array}
\right).$$
The inverse of this matrix then begins
$$\left(
\begin{array}{ccccccc}
 1 & 0 & 0 & 0 & 0 & 0 & 0 \\
 -2 & 1 & 0 & 0 & 0 & 0 & 0 \\
 2 & -2 & 1 & 0 & 0 & 0 & 0 \\
 -2 & 2 & -2 & 1 & 0 & 0 & 0 \\
 2 & -2 & 3 & -3 & 1 & 0 & 0 \\
 -2 & 2 & -4 & 6 & -4 & 1 & 0 \\
 2 & -2 & 5 & -10 & 10 & -5 & 1 \\
\end{array}
\right).$$
Note that we have
$$\left(
\begin{array}{cccccc}
 1 & 0 & 0 & 0 & 0 & 0 \\
 2 & 1 & 0 & 0 & 0 & 0 \\
 2 & 2 & 1 & 0 & 0 & 0 \\
 2 & 2 & 0 & 1 & 0 & 0 \\
 2 & 2 & 0 & 0 & 1 & 0 \\
 2 & 2 & 0 & 0 & 0 & 1 \\
\end{array}
\right) \cdot \left(
\begin{array}{cccccc}
 1 & 0 & 0 & 0 & 0 & 0 \\
 0 & 1 & 0 & 0 & 0 & 0 \\
 0 & 0 & 1 & 0 & 0 & 0 \\
 0 & 0 & 2 & 1 & 0 & 0 \\
 0 & 0 & 3 & 3 & 1 & 0 \\
 0 & 0 & 4 & 6 & 4 & 1 \\
\end{array}
\right)= \left(
\begin{array}{cccccc}
 1 & 0 & 0 & 0 & 0 & 0 \\
 2 & 1 & 0 & 0 & 0 & 0 \\
 2 & 2 & 1 & 0 & 0 & 0 \\
 2 & 2 & 2 & 1 & 0 & 0 \\
 2 & 2 & 3 & 3 & 1 & 0 \\
 2 & 2 & 4 & 6 & 4 & 1 \\
\end{array}
\right).$$
\end{example}
We have
\begin{proposition} The element $\left(\frac{1+x}{1-x}, \frac{1+x}{1-x}; \frac{1}{(1-x)^2}, \frac{x}{1-x}\right)$ of $a\mathcal{R}(2)$ is a pseudo-involution.
\end{proposition}
\begin{proof}
We have
\begin{align*}
b^*&=\left(1; -\frac{1}{(1-x)^2}, \frac{x}{1-x}\right)^{-1} \frac{1+x}{1-x}\\
&=\left(1; -\frac{(1+x)^2}, \frac{x}{1+x}\right)\frac{1+x}{1-x}\\
&=1-\frac{x}{(1+x)^2} \frac{2}{1-\frac{x}{1+x}}\\
&=1-\frac{2}{1+x}=\frac{1-x}{1+x},\end{align*}
which expands to give $1-2,2,-2,\ldots$.

\noindent We then obtain that
\begin{align*}
a^{**}&=\left(1, -\frac{1-x}{1+x}; - \frac{1}{(1+x)^2}, \frac{x}{1+x}\right) \frac{1+x}{1-x}\\
&=1-2x \frac{1-x}{1+x} - \frac{x^2}{(1+x)^2} \frac{2(1+x)}{1+x-x}\\
&=\frac{1-x}{1+x}.\end{align*}
\end{proof}

It is possible to extend this result to higher orders. For instance, we can consider the almost Riordan array of third order defined by $\left(\frac{1+x}{1-x},\frac{1+x}{1-x}, \frac{1+x}{1-x}; \frac{1}{(1-x)^2}, \frac{x}{1-x}\right)$. The corresponding matrix has
\begin{align*}
M_{n,k}&=[x^n] x^k \frac{1+x}{1-x}, \quad k < 3;\\
M_{n,k}&=[x^n] \frac{x^k}{(1-x)^{k-1}}, \quad k \ge 3,\end{align*}
and begins
$$\left(
\begin{array}{cccccccc}
 1 & 0 & 0 & 0 & 0 & 0 & 0 & 0 \\
 2 & 1 & 0 & 0 & 0 & 0 & 0 & 0 \\
 2 & 2 & 1 & 0 & 0 & 0 & 0 & 0 \\
 2 & 2 & 2 & 1 & 0 & 0 & 0 & 0 \\
 2 & 2 & 2 & 2 & 1 & 0 & 0 & 0 \\
 2 & 2 & 2 & 3 & 3 & 1 & 0 & 0 \\
 2 & 2 & 2 & 4 & 6 & 4 & 1 & 0 \\
 2 & 2 & 2 & 5 & 10 & 10 & 5 & 1 \\
\end{array}
\right).$$ The inverse of this matrix begins
$$\left(
\begin{array}{cccccccc}
 1 & 0 & 0 & 0 & 0 & 0 & 0 & 0 \\
 -2 & 1 & 0 & 0 & 0 & 0 & 0 & 0 \\
 2 & -2 & 1 & 0 & 0 & 0 & 0 & 0 \\
 -2 & 2 & -2 & 1 & 0 & 0 & 0 & 0 \\
 2 & -2 & 2 & -2 & 1 & 0 & 0 & 0 \\
 -2 & 2 & -2 & 3 & -3 & 1 & 0 & 0 \\
 2 & -2 & 2 & -4 & 6 & -4 & 1 & 0 \\
 -2 & 2 & -2 & 5 & -10 & 10 & -5 & 1 \\
\end{array}
\right).$$

\section{Involutions in the group of almost Riordan arrays}
Let $R=(g(x), f(x))$ be an ordinary Riordan array of order $2$, which means that $R^2=I$. We call such a Riordan array $R$  an \emph{involution}. Multiplying $R$ with itself, we get
\begin{eqnarray*}
R^2=I \Leftrightarrow (g(x), f(x)) (g(x), f(x))= (1,x),
\end{eqnarray*}
and
\begin{eqnarray} \label{invo}
(g(x)g(f(x)), f(f(x)))=(1,x),
\end{eqnarray}
which gives us $f(f(x))=x$, and $g(f(x))=\frac{1}{g(x)}$.

Searching for involutions among the almost Riordan matrices, we suppose that we have the involution $(a;g,f)$, so we want
\begin{equation*}
(a(x);g(x),f(x)) (a(x);g(x),f(x)) =(1;1,x)
\end{equation*}
which becomes
\begin{equation} \label{invoAlmost}
 ((a(x);g(x),f(x))a(x);g(x)g(f(x)),f(f(x)))=(1;1,x).
\end{equation}
The same conditions as in the case of the involutions of (\ref{invo}), are satisfied for the internal generating functions $g$ and $f$ of (\ref{invoAlmost}), while for the initial column on the left, we have that
\begin{eqnarray} \label{invoAlmost2}
(a(x);g(x),f(x))a(x) &=& a_0a(x)+xg(x) \frac{a(f(x))-a_0}{f(x)}\nonumber\\
&=& a_0a(x)+\frac{xg(x)a(f(x))}{f(x)}-\frac{a_0xg(x)}{f(x)} \nonumber\\
&=& a_0\left( a(x)-\frac{xg(x)}{f(x)} \right)+\frac{xg(x)a(f(x))}{f(x)},
\end{eqnarray}
\noindent which has to be equal to $1$. According to the definition of an almost Riordan array, the generating function of the initial column $a(x)$ is an $\mathbb{F}_0-$ power series, while $g(x) \in \mathcal{F}_0$, and $f(x) \in \mathcal{F}_1$.

A first result is the following.
\begin{proposition} If $(g(x), f(x))$ is an involution in $\mathcal{R}$, then $(1;g(x), f(x))$ is an involution in $a\mathcal{R}(1)$.
\end{proposition}
\begin{proof}
We have $a(x)=1$ and thus
$$a_0\left( a(x)-\frac{xg(x)}{f(x)} \right)+\frac{xg(x)a(f(x))}{f(x)}=1-\frac{xg(x)}{f(x)}+\frac{xg(x)}{f(x)}=1,$$ as required.
\end{proof}
By choosing $a(x)=\frac{xg(x)}{f(x)}$, we get a power series in  $\mathcal{F}_0$, and Eq. (\ref{invoAlmost2}) becomes
\begin{eqnarray*}
a_0\left( a(x)-\frac{xg(x)}{f(x)} \right)+\frac{xg(x)a(f(x))}{f(x)} &=& \frac{xg(x)a(f(x))}{f(x)} \\
&=& \frac{xg(x)\cfrac{f(x)g(f(x))}{f(f(x))}}{f(x)}
\end{eqnarray*}
Applying the conditions of (\ref{invo}), we have that this is equal to $1$. Hence, we have proven the following proposition.
\begin{proposition} Let $(g(x),f(x))$ be an ordinary Riordan involution, then the almost Riordan array $(a(x);g(x),f(x))$ is also an involution if $a(x)=\frac{xg(x)}{f(x)}$.
\end{proposition}
\noindent For the Riordan family of subgroups which are solely defined by their second generating function $f$, $Y_f[r,s,p]$, we only need the condition $f=\bar{f}$, as the $g$ function depends on $f$  \cite{subgroups}.
\begin{corollary}
Let $\upsilon_f[\rho,\sigma,\pi]=\left( \left( \frac{f(x)}{x}\right)^{\rho}((f'(x))^{\sigma} \left( \frac{f(x)-1}{x-1}\right)^{\pi}),f(x)\right)$ be a Riordan involution, then the almost Riordan array $$\left( \left( \frac{f(x)}{x}\right)^{\rho-1}((f'(x))^{\sigma} \left( \frac{f(x)-1}{x-1}\right)^{\pi} ;\left( \frac{f(x)}{x}\right)^{\rho}((f'(x))^{\sigma} \left( \frac{f(x)-1}{x-1}\right)^{\pi}),f(x)\right)$$ is also an involution.
\end{corollary}
\begin{example}
Let $f(x)=-\frac{x}{1+x}$, so $\upsilon_f[1,1,1]$ will be $\left( \frac{f(x)}{x} f'(x) \frac{f(x)-1}{x-1},f(x)\right)$, which corresponds to the matrix
$$\left( \frac{1+2x}{(1+x)^4(1-x)}, -\frac{x}{1+x}\right) = \left(
\begin{array}{cccccccc}
 1 & 0 & 0 & 0 & 0 & 0 & 0 & 0 \\
 -1 & -1 & 0 & 0 & 0 & 0 & 0 & 0 \\
 1 & 2 & 1 & 0 & 0 & 0 & 0 & 0 \\
 1 & -3 & -3 & -1 & 0 & 0 & 0 & 0 \\
 -4 & 2 & 6 & 4 & 1 & 0 & 0 & 0 \\
 10 & 2 & -8 & -10 & -5 & -1 & 0 & 0 \\
 -18 & -12 & 6 & 18 & 15 & 6 & 1 & 0 \\
 30 & 30 & 6 & -24 & -33 & -21 & -7 & -1 \\
\end{array}
\right),$$
\noindent which is an involution, as $f=\bar{f}$. Now, the almost Riordan matrix which contains $\upsilon_f[1,1,1]$ and it is also an involution will be
\begin{eqnarray*}
\left( f'(x) \frac{f(x)-1}{x-1} ;\frac{f(x)}{x} f'(x) \frac{f(x)-1}{x-1},f(x)\right) &=& \left( \frac{1+2x}{(1+x)^3(1-x)}, -\frac{1}{1+x}\right) \\
&=& \left(
\begin{array}{cccccccc}
 1 & 0 & 0 & 0 & 0 & 0 & 0 & 0 \\
 0 & 1 & 0 & 0 & 0 & 0 & 0 & 0 \\
 0 & -1 & -1 & 0 & 0 & 0 & 0 & 0 \\
 2 & 1 & 2 & 1 & 0 & 0 & 0 & 0 \\
 -3 & 1 & -3 & -3 & -1 & 0 & 0 & 0 \\
 6 & -4 & 2 & 6 & 4 & 1 & 0 & 0 \\
 -8 & 10 & 2 & -8 & -10 & -5 & -1 & 0 \\
 12 & -18 & -12 & 6 & 18 & 15 & 6 & 1 \\
\end{array}
\right)
\end{eqnarray*}
\end{example}

We note that in the general case, the almost Riordan array $\left(\frac{x g(x)}{f(x)}; g(x), f(x)\right)$ is in fact a Riordan array. It coincides with the Riordan array $\left(\frac{xg(x)}{f(x)}, f(x)\right)$. Iterating, we see that if $(g, f)$ is an involution in $\mathcal{R}$, then so is
$\left(\frac{x^n}{f(x)^n}g(x), f(x)\right)$.
\begin{corollary} Let $(g(x), f(x))$ be a Riordan involution. Then $\left(\left(\frac{x}{f(x)}\right)^n g(x), f(x)\right)$ is a Riordan involution. \end{corollary}

\section{Quasi-involutions in the group of almost Riordan arrays}

We consider the quasi-involution  \cite{quasi}
$$(g(x), xg(x))=\left(\frac{1-x^2-\sqrt{1-6x^2+x^4}}{2x^2}, \frac{1-x^2-\sqrt{1-6x^2+x^4}}{2x}\right).$$
The corresponding matrix begins
$$\left(
\begin{array}{ccccccccc}
 1 & 0 & 0 & 0 & 0 & 0 & 0 & 0 & 0 \\
 0 & 1 & 0 & 0 & 0 & 0 & 0 & 0 & 0 \\
 2 & 0 & 1 & 0 & 0 & 0 & 0 & 0 & 0 \\
 0 & 4 & 0 & 1 & 0 & 0 & 0 & 0 & 0 \\
 6 & 0 & 6 & 0 & 1 & 0 & 0 & 0 & 0 \\
 0 & 16 & 0 & 8 & 0 & 1 & 0 & 0 & 0 \\
 22 & 0 & 30 & 0 & 10 & 0 & 1 & 0 & 0 \\
 0 & 68 & 0 & 48 & 0 & 12 & 0 & 1 & 0 \\
 90 & 0 & 146 & 0 & 70 & 0 & 14 & 0 & 1 \\
\end{array}
\right).$$
The initial column of this matrix is given by the aerated large Schroeder numbers
$$1, 0, 2, 0, 6, 0, 22, 0, 90, 0, 394, 0, 1806, 0, 8558,\ldots.$$ The inverse of this matrix begins
$$\left(
\begin{array}{ccccccccc}
 1 & 0 & 0 & 0 & 0 & 0 & 0 & 0 & 0 \\
 0 & 1 & 0 & 0 & 0 & 0 & 0 & 0 & 0 \\
 -2 & 0 & 1 & 0 & 0 & 0 & 0 & 0 & 0 \\
 0 & -4 & 0 & 1 & 0 & 0 & 0 & 0 & 0 \\
 6 & 0 & -6 & 0 & 1 & 0 & 0 & 0 & 0 \\
 0 & 16 & 0 & -8 & 0 & 1 & 0 & 0 & 0 \\
 -22 & 0 & 30 & 0 & -10 & 0 & 1 & 0 & 0 \\
 0 & -68 & 0 & 48 & 0 & -12 & 0 & 1 & 0 \\
 90 & 0 & -146 & 0 & 70 & 0 & -14 & 0 & 1 \\
\end{array}
\right).$$
We say that when an aerated matrix has an inverse with the same elements as the original matrix, except that the signs change on alternate non-zero diagonals, then the matrix is a quasi-involution. A Riordan array $(g(x^2), xg(x^2))$ is a quasi-involution if its inverse is given by $(g(-x^2), xg(-x^2))$. This will be the case if and only if the $\Omega_{1,1}$-sequence of the array satisfies $\Omega_{1,1}(x)=\frac{1}{\Omega_{1,1}(-x)}$ \cite{quasi}. For the matrix above, we have $\Omega_{1,1}=\frac{1+x}{1-x}$.

\begin{proposition} If we replace the initial column of the Riordan quasi-involution matrix
$$\left(\frac{1-x^2-\sqrt{1-6x^2+x^4}}{2x^2}, \frac{1-x^2-\sqrt{1-6x^2+x^4}}{2x}\right)$$  with the sequence
$$1, 0, \alpha, 0, \beta, 0, \gamma, 0, \delta, 0,\epsilon, 0, \zeta, 0,\eta,\ldots$$
where $\alpha, \gamma, \epsilon, \eta, \ldots$ are arbitrary, and we have
$\beta=3 \gamma, \delta=-32 \alpha+7 \gamma, \zeta=1200 \alpha-224 \gamma+11 \epsilon,\ldots$,
then we get a quasi-involution in the group of almost Riordan arrays.
\end{proposition}
Note that when $\alpha=2, \gamma=22, \epsilon=394, \ldots$, we obtain the original matrix.
\begin{proof}
The ``interior matrix'' that begins
$$\left(
\begin{array}{cccccccc}
 1 & 0 & 0 & 0 & 0 & 0 & 0 & 0 \\
 0 & 1 & 0 & 0 & 0 & 0 & 0 & 0 \\
 4 & 0 & 1 & 0 & 0 & 0 & 0 & 0 \\
 0 & 6 & 0 & 1 & 0 & 0 & 0 & 0 \\
 16 & 0 & 8 & 0 & 1 & 0 & 0 & 0 \\
 0 & 30 & 0 & 10 & 0 & 1 & 0 & 0 \\
 68 & 0 & 48 & 0 & 12 & 0 & 1 & 0 \\
 0 & 146 & 0 & 70 & 0 & 14 & 0 & 1 \\
\end{array}
\right),$$ is the Riordan array $(g(x)^2, xg(x))$ which is known to be a quasi-involution since $(g(x), xg(x))$ is.
Thus the matrix above is an almost Riordan array of the first order.
We must now only show that the signs of the elements of the initial column of the inverse of the matrix described above agree with those of the original matrix, except for changes of sign on alternate non-zero positions. Now the elements above $\alpha, \gamma, \epsilon$ have been calculated iteratively to make this so.
\end{proof}
Thus we have
$$\left(
\begin{array}{ccccccccc}
 1 & 0 & 0 & 0 & 0 & 0 & 0 & 0 & 0 \\
 0 & 1 & 0 & 0 & 0 & 0 & 0 & 0 & 0 \\
 \alpha  & 0 & 1 & 0 & 0 & 0 & 0 & 0 & 0 \\
 0 & 4 & 0 & 1 & 0 & 0 & 0 & 0 & 0 \\
 3 \alpha  & 0 & 6 & 0 & 1 & 0 & 0 & 0 & 0 \\
 0 & 16 & 0 & 8 & 0 & 1 & 0 & 0 & 0 \\
 \gamma  & 0 & 30 & 0 & 10 & 0 & 1 & 0 & 0 \\
 0 & 68 & 0 & 48 & 0 & -12 & 0 & 1 & 0 \\
 7 \gamma -32 \alpha  & 0 & 146 & 0 & 70 & 0 & 14 & 0 & 1 \\
\end{array}
\right)^{-1}$$
$$=\left(
\begin{array}{ccccccccc}
 1 & 0 & 0 & 0 & 0 & 0 & 0 & 0 & 0 \\
 0 & 1 & 0 & 0 & 0 & 0 & 0 & 0 & 0 \\
 -\alpha  & 0 & 1 & 0 & 0 & 0 & 0 & 0 & 0 \\
 0 & -4 & 0 & 1 & 0 & 0 & 0 & 0 & 0 \\
 3 \alpha  & 0 & -6 & 0 & 1 & 0 & 0 & 0 & 0 \\
 0 & 16 & 0 & -8 & 0 & 1 & 0 & 0 & 0 \\
 -\gamma  & 0 & 30 & 0 & -10 & 0 & 1 & 0 & 0 \\
 0 & -68 & 0 & 48 & 0 & -12 & 0 & 1 & 0 \\
  7 \gamma -32 \alpha  & 0 & -146 & 0 & 70 & 0 & -14 & 0 & 1 \\
\end{array}
\right).$$
Note that this proof is valid only for finite Riordan arrays.
\begin{example} The Riordan array $$\left(\frac{1}{(1-4x^2)^{\frac{3}{2}}}, \frac{x}{\sqrt{1-4x^2}}\right),$$ which begins
$$\left(
\begin{array}{ccccccc}
 1 & 0 & 0 & 0 & 0 & 0 & 0 \\
 0 & 1 & 0 & 0 & 0 & 0 & 0 \\
 6 & 0 & 1 & 0 & 0 & 0 & 0 \\
 0 & 8 & 0 & 1 & 0 & 0 & 0 \\
 30 & 0 & 10 & 0 & 1 & 0 & 0 \\
 0 & 48 & 0 & 12 & 0 & 1 & 0 \\
 140 & 0 & 70 & 0 & 14 & 0 & 1 \\
\end{array}
\right),$$ is a quasi-involution. Its inverse begins
$$\left(
\begin{array}{ccccccc}
 1 & 0 & 0 & 0 & 0 & 0 & 0 \\
 0 & 1 & 0 & 0 & 0 & 0 & 0 \\
 -6 & 0 & 1 & 0 & 0 & 0 & 0 \\
 0 & -8 & 0 & 1 & 0 & 0 & 0 \\
 30 & 0 & -10 & 0 & 1 & 0 & 0 \\
 0 & 48 & 0 & -12 & 0 & 1 & 0 \\
 -140 & 0 & 70 & 0 & -14 & 0 & 1 \\
\end{array}
\right).$$
If we now adjoin the sequence that begins
$$1, 0,r,0,4r, 0, s,0,8s-64r, 0, t, 0,8t-64r, 0,u,0,4(768r-80s+3t),0,v,0,\ldots$$ to this matrix, to get a matrix that begins
$$\left(
\begin{array}{ccccccc}
 1 & 0 & 0 & 0 & 0 & 0 & 0 \\
 0 & 1 & 0 & 0 & 0 & 0 & 0 \\
 r & 0 & 1 & 0 & 0 & 0 & 0 \\
 0 & 6 & 0 & 1 & 0 & 0 & 0 \\
 4 r & 0 & 8 & 0 & 1 & 0 & 0 \\
 0 & 30 & 0 & 10 & 0 & 1 & 0 \\
 s & 0 & 48 & 0 & 12 & 0 & 1 \\
\end{array}
\right)$$ then we obtain an almost Riordan array which is a quasi-involution. The inverse then begins
$$\left(
\begin{array}{ccccccc}
 1 & 0 & 0 & 0 & 0 & 0 & 0 \\
 0 & 1 & 0 & 0 & 0 & 0 & 0 \\
 -r & 0 & 1 & 0 & 0 & 0 & 0 \\
 0 & -6 & 0 & 1 & 0 & 0 & 0 \\
 4 r & 0 & -8 & 0 & 1 & 0 & 0 \\
 0 & 30 & 0 & -10 & 0 & 1 & 0 \\
 -s & 0 & 48 & 0 & -12 & 0 & 1 \\
\end{array}
\right).$$
In effect, the terms $r, s, t, u, v,\ldots$ are calculated iteratively precisely to give the quasi-involution property. For instance, if we set the terms $r, s, t, \ldots$ to $1$ then we obtain an initial column
$$1, 0, 1, 0, 4, 0, 1, 0, -56, 0, 1, 0, 2764, 0, 1, 0, -250736, 0, 1,\ldots.$$

Note that alternatively, we could have started with the quasi-involution $$\left(\frac{1}{1-4x^2}, \frac{x}{\sqrt{1-4x^2}}\right)$$ which begins
$$\left(
\begin{array}{cccccccc}
 1 & 0 & 0 & 0 & 0 & 0 & 0 & 0 \\
 0 & 1 & 0 & 0 & 0 & 0 & 0 & 0 \\
 4 & 0 & 1 & 0 & 0 & 0 & 0 & 0 \\
 0 & 6 & 0 & 1 & 0 & 0 & 0 & 0 \\
 16 & 0 & 8 & 0 & 1 & 0 & 0 & 0 \\
 0 & 30 & 0 & 10 & 0 & 1 & 0 & 0 \\
 64 & 0 & 48 & 0 & 12 & 0 & 1 & 0 \\
 0 & 140 & 0 & 70 & 0 & 14 & 0 & 1 \\
\end{array}
\right).$$ Then we can  substitute the above sequence for the initial column.
\end{example}

\end{document}